\documentclass[12pt, notitlepage]{amsart}
\usepackage{amsmath, amsthm, amsfonts, amssymb, cite}
\usepackage{wrapfig}
\usepackage{stackrel}
\usepackage{color}
\usepackage{graphicx}
\usepackage{float}
\usepackage{enumitem}
\usepackage{mathtools}
\usepackage{multicol}
\usepackage[normalem]{ulem}
\usepackage[utf8]{inputenc}

\newcommand{\lt}{\left}
\newcommand{\rt}{\right}
\newcommand{\bpm}{\begin{pmatrix}}
\newcommand{\epm}{\end{pmatrix}}
\newcommand{\bsm}{\lt(\begin{smallmatrix}}
\newcommand{\esm}{\end{smallmatrix}\rt)}
\newcommand{\beq}{\begin{equation}}
\newcommand{\eeq}{\end{equation}}
\newcommand{\bmat}{\begin{matrix*}}
\newcommand{\emat}{\end{matrix*}}

\newcommand{\Z}{\ensuremath{\mathbb{Z}}}

\newcommand{\Q}{\ensuremath{\mathbb{Q}}}
\newcommand{\R}{\ensuremath{\mathbb{R}}}

\renewcommand{\a}{\mathfrak{a}}
\renewcommand{\b}{\mathfrak{b}}
\let\turc\c
\usepackage{etoolbox}
\robustify\turc 
\renewcommand{\c}{\mathfrak{c}}

\DeclareMathOperator{\SL}{SL}

\newcommand{\inv}{^{-1}}

\renewcommand{\Im}{\operatorname{Im}}

\usepackage[usenames,dvipsnames]{xcolor}

\usepackage[inner=1.0in,outer=1.0in,bottom=1.0in, top=1.0in]{geometry}

\usepackage{amsmath, mathtools}

\newcommand{\kloostermanmodulus}{\gamma}

\newcommand{\mr}{\ensuremath{\mathbb R}}

\newcommand{\mymod}{\ensuremath{\negthickspace \negmedspace \pmod}}
\newcommand{\shortmod}{\ensuremath{\negthickspace \negthickspace \negthickspace \pmod}}

\newcommand{\sumstar}{\sideset{}{^*}\sum}

\theoremstyle{plain}            
        \newtheorem{mytheo}{Theorem} [section]
        
        \newtheorem{myprop}[mytheo]{Proposition}
        \newtheorem{mycoro}[mytheo]{Corollary}
     \newtheorem{mylemma}[mytheo]{Lemma}
        \newtheorem{mydefi}[mytheo]{Definition}

\numberwithin{equation}{section}
\numberwithin{figure}{section}

\title{Kloosterman sums and Fourier coefficients of Eisenstein series} 

\author{Eren Mehmet K{\i}ral, Matthew P. Young}

\address{Department of Mathematics \\
	  Texas A\&M University \\
	  College Station \\
	  TX 77843-3368 \\
		U.S.A.}
\email{ekiral@tamu.edu}		
\email{myoung@math.tamu.edu}

\thanks{This material is based upon work of M.Y. supported by the National Science Foundation under agreement No. DMS-1401008.  Any opinions, findings and conclusions or recommendations expressed in this material are those of the authors and do not necessarily reflect the views of the National Science Foundation. \\
 The first author was partially supported by an AMS-Simons travel grant.}

\begin{document}

\begin{abstract}
We derive explicit formulas for some Kloosterman sums on $\Gamma_0(N)$, and for the Fourier coefficients of Eisenstein series attached to arbitrary cusps, around a general Atkin-Lehner cusp.
\end{abstract}

\maketitle

\section{Motivation}
Many problems in analytic number theory rely on the spectral theory of automorphic forms for $GL_2$.  For instance, it is desirable to obtain cancellation in sums of Kloosterman sums, a goal first achieved by Kuznetsov \cite{KuznetsovSumOfKloosterman} for $\sum_{c \leq X} \frac{S(m,n;c)}{c}$ 
with the key tool being his famous spectral decomposition of sums of Kloosterman sums.  For arithmetical applications, one typically encounters Kloosterman sums with additional constraints.
As an example, for the problem of estimating the fifth moment of modular $L$-functions in the level aspect \cite{KiralYoung5thMoment}, we encountered
\begin{equation}
\label{eq:sumofkloostermanOriginalExample}
		\sum_{\substack{(c,N) = 1\\c \equiv 0 \shortmod{q}}} S(\overline{N} m,n;c)f(c),
	\end{equation}
	where $f$ is a smooth function on the positive reals with sufficient decay.  
	
	Deshouillers and Iwaniec \cite{DeshouillersIwaniec} greatly generalized Kuznetsov's arguments and in particular calculated $S_{\a  \b}(m,n;c)$ for a number of pairs of cusps $\a, \b$ for the modular group $\Gamma_0(N)$.  
	In an unpublished manuscript, Motohashi \cite{motohashi2007riemann} showed
	that \eqref{eq:sumofkloostermanOriginalExample}
	can be realized as the sum of Kloosterman sums associated to the pair of cusps $\infty, 1/q$ for the group $\Gamma_0(Nq)$.
	Moreover, the scaling matrix associated to the cusp $1/q$ may be chosen to be an Atkin-Lehner operator.
	One goal of this paper is to derive a much more general version of \eqref{eq:sumofkloostermanOriginalExample}, and to systematically choose Atkin-Lehner operators as scaling matrices.  Such a choice is important because if $f$ is a newform, then it is an eigenfunction of the Atkin-Lehner operators, and so the Fourier coefficients at a cusp $\a$ will be the same, up to a sign, as those at the cusp $\infty$, provided $\a$ is equivalent to $\infty$ under an Atkin-Lehner operator.  This property was crucially used in \cite{KiralYoung5thMoment} when studying \eqref{eq:sumofkloostermanOriginalExample}.  Our Kloosterman sum formula appears as Theorem \ref{thm:KloostermanCalculation} below.
	
Another problem in this theme is the explicit evaluation of Fourier coefficients of Eisenstein series attached to general cusps $\a$ on $\Gamma_0(N)$.  These quantities appear on the spectral side of the Bruggeman-Kuznetsov formula, and for some applications (e.g. \cite{KiralYoung5thMoment}) one needs rather precise information on these Fourier coefficients.  In this paper, we provide the evaluation of these Fourier coefficients, for which see Theorem \ref{thm:EisensteinFourierCoefficientFormulaDirichletCharacters} below.

This is a companion paper to \cite{KiralYoung5thMoment} which relies on the results in this paper.

\section{Kloosterman sums}	
	
\subsection{Definitions}	

We mostly follow the notation of \cite{IwaniecSpectralBook}. 
	Let $N$ be a positive integer and $\Gamma = \Gamma_0(N)$.
       An element $\a \in \mathbb{P}^1(\Q)$ is called a cusp. Two cusps $\a$ and $\a'$ are  equivalent under $\Gamma$ if there is a $\gamma \in \Gamma$ satisfying $\a' = \gamma \a.$
            Let $\a$ be a cusp and 
                $\Gamma_\a = \{ \gamma \in \Gamma : \gamma \a = \a\}$
            be the stabilizer of the cusp $\a$ in $\Gamma$. 
            A matrix $\sigma_\a \in \SL_2(\R)$, satisfying      
            \begin{equation} \label{eq:scalingMatrixProperties}
                \sigma_\a \infty = \a, \quad \text{ and } \quad \sigma_\a\inv \Gamma_\a \sigma_\a = \left\{ \pm \left(\begin{smallmatrix} 1 & n\\ 0&1 \end{smallmatrix} \right) : n \in \Z\right\}
            \end{equation}
            is called a scaling matrix for the cusp $\a$. 
            
        {\bf Remarks.} For two equivalent cusps $\a$ and $\a' = \gamma \a$, with $\gamma \in \Gamma$, the stabilizers of the cusps are conjugate subgroups in $\Gamma$, namely $\Gamma_{\a'} = \gamma \Gamma_\a \gamma\inv$. 
        Furthermore, $\sigma_{\a'} = \gamma \sigma_\a$ 
        is a scaling matrix for the cusp $\a'$. 
                 The matrix $\sigma_\a$ is not uniquely defined by the above properties: for any $\alpha \in \R$, $$\sigma_\a\bsm 1&\alpha\\ 0&1\esm$$ also satisfies
        \eqref{eq:scalingMatrixProperties}. The choice of the scaling matrix $\sigma_\a$ is important in what follows.
        
        \begin{mydefi}
             Let $f$ be a Maass form for the group $\Gamma$. The Fourier coefficients of $f$ at a cusp $\a$, \emph{relative to a choice of $\sigma_\a$}, and denoted $\rho_f(\sigma_{\a}, n)$, are defined by 
             \begin{equation}
              \label{eq:FourierExpansionMaassForm}
              f(\sigma_\a z) = \sum_{n \neq 0} \rho_f(\sigma_\a, n) e(nx)
                  W_{0, i t_j}(4 \pi |n| y),
             \end{equation}
             where
             \begin{equation*}
              W_{0,it_j}(4 \pi y) = 2 \sqrt{|y|} K_{i t_j} (2 \pi y).
             \end{equation*}

        \end{mydefi}
  {\bf Remark.} If $\sigma_\a$ is replaced with $\sigma_\a \bsm 1&\alpha \\ 0&1 \esm$, then the Fourier coefficients relative to 
        the new scaling matrix are given by 
        \begin{equation*}
              \rho_f(\sigma_a \bsm 1&\alpha\\ 0&1\esm, n) = e(n\alpha) \rho_f( \sigma_\a, n).
         \end{equation*} 
         If the Fourier coefficients of $f$ at a cusp $\a$ are multiplicative with respect to $\sigma_\a$, 
         then for another scaling matrix as above, the Fourier coefficients will typically not be multiplicative.  We found the reference \cite{GoldfeldHundleyLee2015Multiplicative} useful for its discussions in this context.
         
         
         Fourier coefficients at equivalent cusps, however, behave more predictably. If $\a' = \gamma \a$ and we choose the scaling matrix  
         as $\sigma_{\a'} = \gamma \sigma_{\a}$, then due to $\Gamma$-invariance of $f$, we have
         \begin{equation*}
              \rho_f(\sigma_{\a'}, n) = \rho_f(\sigma_\a, n).
         \end{equation*}
 When the scaling matrix $\sigma_{\a}$ is understood, we may write $\rho_{f}(\sigma_{\a}, n) = \rho_{\a, f}(n)$. 

        We now define Kloosterman sums with respect to a pair of cusps, and for general nebentypus.
Let $\chi$ be a Dirichlet character modulo $N$. We extend $\chi$ to $\Gamma$ via
	\begin{align*}
		\chi: \Gamma &\longrightarrow S^1\\
			\gamma = \bsm a&b\\Nc&d \esm &\mapsto \chi(d).
	\end{align*}
	If $\chi$ is an even character, it can be seen as a multiplier system on $\Gamma$ with weight $0$. Let $\lambda_\a$ be defined by $\sigma_\a\inv \lambda_\a \sigma_\a = \bsm 1&1\\0&1 \esm$. We say that the cusp $\a$ is \emph{singular} for $\chi$ if $\chi(\lambda_\a) = 1$. 
	\begin{mydefi}
	\label{def:KloostermanSum}
              For $\a$ and $\b$ singular cusps for $\chi$, we define the Kloosterman sum associated to $\a,\b$ and $\chi$ with modulus $c$ as
            \begin{equation}\label{eq:kloostermanDefinition}
		S_{\a \b}  (m,n;c;\chi) =
		\sum_{\gamma = \bsm a & b\\ c &d \esm \in \Gamma_\infty \backslash \sigma_\a\inv \Gamma \sigma_\b / \Gamma_\infty }
		\overline{\chi(\sigma_\a \gamma \sigma_\b\inv)} e\left(\frac{am + dn }{c} \right).
            \end{equation}
        \end{mydefi}

{\bf Remarks.}  This definition of the Kloosterman sum slightly differs from that of \cite[(2.23)]{IwaniecSpectralBook}, in that the roles of $m$ and $n$ are reversed, but agrees with \cite[(3.13)]{IwaniecClassicalBook}.
 
	From Definition \ref{def:KloostermanSum} 
 	, and the fact that $\gamma' = -\gamma\inv$ parametrizes the double coset $\Gamma_\infty \backslash \sigma_\b\inv \Gamma \sigma_\a/ \Gamma_\infty$ as $\gamma$ runs over representatives of the double cosets in $\Gamma_\infty \backslash  \sigma_\a\inv \Gamma \sigma_\b /\Gamma_\infty$, we have
	\begin{equation}\label{eq:switchCusps}
		S_{\a  \b}(m,n;c;\chi) = \overline{S_{\b  \a}(n,m;c;\chi)}.
	\end{equation}
	This corrects a formula in \cite[p.48]{IwaniecSpectralBook} which omitted the complex conjugation. 

	\begin{mydefi} The set of allowed moduli is defined as
            \begin{equation}\label{eq:allowedKloostermanModuli}
                \mathcal{C}_{\a \b} = \left\{\kloostermanmodulus > 0: \bsm * &*\\ \kloostermanmodulus &* \esm \in \sigma_\a\inv \Gamma \sigma_\b \right\}.
            \end{equation}
        \end{mydefi}
        Notice that if $\kloostermanmodulus \notin \mathcal{C}_{\a \b}$ the Kloosterman sum of modulus $\kloostermanmodulus$ is an empty sum. 

	The definition \eqref{eq:kloostermanDefinition} is a natural one, as it occurs in the Fourier expansion of the Poincar\'{e} series which is defined as 
	\begin{equation}\label{eq:poincareSeriesDefinition}
		P_n^\a(z,s;\chi)
		= \sum_{\gamma \in \Gamma_\a \backslash \Gamma} 
		\overline{\chi(\gamma)} \Im(\sigma_\a\inv \gamma z)^s e(n \sigma_\a\inv \gamma z).
	\end{equation}
	Here  $P_n^\a(z,s;\chi)$ is $\Gamma$-automorphic of nebentypus $\chi$
	(to be clear, it transforms by $f(\gamma z) = \chi(\gamma) f(z)$).  

	{\bf Remark.} The Kloosterman sum associated to the pair of cusps $\a,\b$ depends on the choice of pair of scaling matrices $\sigma_\a$ and $\sigma_\b$
	(so it might be better to denote it as $S_{\sigma_\a, \sigma_\b}(m,n;c)$).  If one changes the choice of the scaling matrix, the Kloosterman sum
	also changes by 
	\begin{equation}
              S_{\sigma_\a \bsm 1&\alpha\\ 0&1\esm, \sigma_b \bsm 1& \beta \\ 0&1 \esm } (m,n;c;\chi) = e(-\alpha m +  \beta n) S_{\sigma_\a, \sigma_\b} (m,n;c; \chi). 
         \end{equation}
         This corrects a formula of \cite[p.48]{IwaniecSpectralBook} which has $\alpha$ in place of our $-\alpha$.
         However, changing the cusp $\a$ to an equivalent one does not alter the Kloosterman sum, if one also changes the scaling matrices accordingly. Indeed, if 
         $\a' = \gamma_1 \a$ and $\b' = \gamma_2 \b$ for $\gamma_1,\gamma_2 \in \Gamma$, then
         \begin{equation*}
              S_{\sigma_{\a},\sigma_\b}(m,n;c;\chi) = S_{\gamma_1\sigma_{\a}, \gamma_2\sigma_{\a} }(m,n;c;\chi).
         \end{equation*}
 
         If one applies the Bruggeman-Kuznetsov formula 
         to sums of Kloosterman sums associated to the choice of scaling matrices $\sigma_\a, \sigma_\b$, 
         then the Fourier coefficients at cusps $\a, \b$ appearing on the spectral side must also be computed using the same scaling matrices.

\subsection{Atkin-Lehner cusps and scaling matrices}\label{subsec:AtkinLehnerScaling}

        Assume that $N = rs$ with $(r,s) = 1$. We then call a cusp of the form $\mathfrak{a} = 1/r$ an \emph{Atkin-Lehner cusp}. 
        The Atkin-Lehner cusps are precisely those that are equivalent to $\infty$ under an Atkin-Lehner operator, justifying their name.
        In Proposition \ref{prop:stabilizerScaling} below, we calculate the stablizer of a general cusp $\a$, which when specialized to an Atkin-Lehner cusp shows that $\Gamma_{1/r}$ is generated by $\pm (\begin{smallmatrix} 1 - N & s \\ -r N & 1 + N \end{smallmatrix})$.
        In particular we see that any  even Dirichlet character $\chi \pmod N$ is singular at every Atkin-Lehner cusp.
        
        \begin{mydefi}
             Let $N = rs$ with $(r,s) =1$ as above, and let $W = W_s$ be the Atkin-Lehner operator defined by
             \begin{equation}\label{eq:AtkinLehnerDefinition}
                W= W_s = \bpm xs& y\\ zN & ws \epm,
             \end{equation}
             with $\det(W_s) = s$. On automorphic forms of (even) weight $\kappa$, the Atkin-Lehner operator is defined by $(f\vert_{W})(z) = \det(W)^{\frac{\kappa}{2}} j(W,z)^{-\kappa} f(Wz)$. 
        \end{mydefi}
        Any two choices of Atkin-Lehner operators $W_s$, for the same value of $s$, differ by an element of $\Gamma_0(N)$.  Therefore, if the nebentypus $\chi$ is trivial then $f \vert_{W_{s}}$ is independent of the choices of $x,y,z,w$.  For general nebentypus, one may ensure that $f\vert_{W_s}$ is independent of the choices under the assumptions
         $x \equiv 1 \pmod{r}$ and $z \equiv 1 \pmod{s}$. 

        The matrix can be normalized to have determinant $1$, without changing the operator, via
        \begin{equation*}
             \frac{1}{\sqrt{s}} W_s = \bpm x\sqrt{s}& y/\sqrt{s}\\ z r\sqrt{s} & w \sqrt{s}\epm = \bpm x & y\\ zr & ws \epm \bpm \sqrt{s} & 0 \\ 0& 1/\sqrt{s} \epm.
        \end{equation*}
        Here the determinant condition is $xws -rzy = 1$. We have the freedom to choose $x = z = 1$, and then if $\overline{s}$ is
        any integer that satisfies $s \overline{s} \equiv 1 \pmod {r}$, put $w = \overline{s}$ and $y = (\overline{s}s - 1)/r$. Therefore the matrix
         \begin{equation} 
         \label{eq:AtkinLehnerScaling}
             \sigma_{1/r} = \tau_{r} \nu_s  \quad \text{ with } \quad \tau_{r} = \bpm 1& (\overline{s}s -1)/r \\ r & \overline{s}s\epm, \quad \nu_{s} = \bpm \sqrt{s} & 0\\ 0&1/\sqrt{s}\epm,
         \end{equation}
         is an acceptable choice for an Atkin-Lehner operator $W_s$.  Note $\tau_r \in \Gamma_0(r)$ (in particular, it has integer entries and determinant $1$).  From the theory of Atkin-Lehner operators, a newform $f$ of weight $0$ and trivial nebentypus will satisfy 
         \begin{equation}\label{eq:AtkinLehnerEigenfunction}
             f(\sigma_{1/r} z) = \eta_s(f) f(z),
         \end{equation}
         where $f \vert_{W_s} = \eta_{s}(f) f$ with $\eta_s(f) = \pm 1$. 

        One may check directly that $\sigma_{1/r}$ also satisfies the conditions in \eqref{eq:scalingMatrixProperties}, i.e. it is a scaling matrix for the cusp $1/r$. 
        If $f$ is a Maass form satisfying \eqref{eq:AtkinLehnerEigenfunction} then its Fourier coefficients at the cusp $\a$ with respect to the choice 
        \eqref{eq:AtkinLehnerScaling} for its scaling matrix has the Fourier coefficients 
        \begin{equation*}
             \rho_f(\sigma_{1/r} , n) = \pm 
             \rho_{\infty, f}(n)             
             .
        \end{equation*}
  For an application in \cite{KiralYoung5thMoment}, we need a more general version of this, which takes into account the translates $f{\vert_{\ell}}(z) := f(\ell z)$.
\begin{mylemma}
\label{lemma:FourierExpansionDifferentCusps}
Suppose $\mathfrak{a}$ is an Atkin-Lehner cusp of $\Gamma_0(N)$, and $f^*$ is a newform of trivial nebentypus and level $M$ with $LM=N$.  Then the set of lists of Fourier coefficients 
\{$(\rho_{\mathfrak{a}  f^*\vert_{\ell}}(n))_{n \in \mathbb{N}}$ : $\ell | L$\}
is, up to signs, the same as the set of lists of Fourier coefficients 
\{$(\rho_{\mathfrak{\infty}  f^*\vert_{\ell}}(n))_{n \in \mathbb{N}}$ : $\ell | L$\}.
\end{mylemma}
\begin{proof}
Suppose that $p$ is prime, $p|N$, $p^{\alpha} || N$, $p^{\beta} || L$ (so $p^{\alpha-\beta} || M$), and $p^{\gamma} || \ell$.  Let $W_{p^{\alpha}}$ be the Atkin-Lehner involution for $\Gamma_0(N)$ for the prime $p$.  Then the key fact we need is
\begin{equation}
\label{eq:f*slashing}
(f^* \vert_{\ell}) \vert_{W_{p^{\alpha}}} = \eta_{p^{\alpha-\beta}}(f^*) \thinspace f^* \vert_{\ell'}, 
\end{equation}
where 
$\ell'$ is defined by $\ell' = p^{\beta-\gamma} h$ where $\ell = p^{\gamma} h$, (so $(h,p)= 1$).  Note that the map $\ell \rightarrow \ell'$ is an involution, permuting the divisors of $L$.  Taking \eqref{eq:f*slashing} for granted for a moment, we may complete the proof of Lemma \ref{lemma:FourierExpansionDifferentCusps}, by noting that the Fourier coefficients of $f^* \vert_{\ell}$ at an Atkin-Lehner cusp $\mathfrak{a}$ are equal to the Fourier coefficients of $(f^* \vert_{\ell} )\vert_{W_D}$ for some Atkin-Lehner involution with $D|N$, which is a composition of $W_{p^{\alpha}}$'s.  The lemma follows from repeated usage of \eqref{eq:f*slashing}.

Now we prove \eqref{eq:f*slashing}.  First suppose $\gamma \leq \beta/2$, and let $\ell' = p^{\beta-2\gamma} \ell = p^{\beta-\gamma} h$.    Then by \cite[Lemma 26]{Atkin-Lehner},
\begin{equation}
\label{eq:AtkinLehnerLemma26Formula}
 (f^* \vert_{\ell}) \vert_{W_{p^{\alpha}}} = (f^* \vert_{W'_{p^{\alpha-\beta}}}) \vert_{{\ell'}} = \eta_{p^{\alpha-\beta}}(f^*) f^* \vert_{\ell'},
\end{equation}
where $W'_{p^{\alpha-\beta}}$ is the Atkin-Lehner involution on $\Gamma_0(M)$ (technically, they worked with holomorphic forms but their proof works equally well for Maass forms).  This proves the claim under the condition $\gamma \leq \beta/2$.  If $\gamma > \beta/2$, then one may reverse the roles of $\ell$ and $\ell'$ and apply $W_{p^{\alpha}}$ to both sides of \eqref{eq:AtkinLehnerLemma26Formula} to give the result.
\end{proof}
        
\subsection{Kloosterman sums using Atkin-Lehner scaling}
        
        \begin{myprop} \label{prop:KloostermanSumsAtInfinityAndAtkinLehnerCusp}
             Let $N = rs$ with $(r,s) = 1$, and choose $\sigma_{1/r}$ as in \eqref{eq:AtkinLehnerScaling}. Then the set of allowed moduli for the pair 
             of cusps $\infty, \tfrac{1}{r}$ is 
             \begin{equation} \label{eq:infinity1rAllowedModuli}
                  \mathcal{C}_{\infty, 1/r} = \left\{\kloostermanmodulus = c \sqrt{s} > 0 : c \equiv 0 \mymod{r}, \thinspace (c,s) = 1\right\},
             \end{equation}
             and for such $\kloostermanmodulus = c\sqrt{s} \in \mathcal{C}_{\infty,1/r}$, the Kloosterman sum to modulus $\kloostermanmodulus$ is given by
             \begin{equation}
                  S_{\infty,1/r}  (m,n;c\sqrt{s}) = S(\overline{s} m,n;c),
             \end{equation}
            where the $S(a,b;c)$ on the right denotes the ordinary Kloosterman sum.
        \end{myprop}
        
        {\bf Remark.} This is the same Kloosterman sum as  \cite[(14.8)]{motohashi2007riemann}, but differs from the computation in 
        \cite[Section 4.2]{IwaniecClassicalBook} by the presence of an additive character. This difference is due to the differing choices of the scaling matrices. Motohashi's 
        choice of scaling matrix corresponds to the Atkin-Lehner operators as in Section \ref{subsec:AtkinLehnerScaling} above. Proposition\ref{prop:KloostermanSumsAtInfinityAndAtkinLehnerCusp} is a special case of the following more general Theorem \ref{thm:KloostermanCalculation}, which 
evaluates the Kloosterman sum that 
        is associated to the pair of Atkin-Lehner cusps $1/r_1$ and $1/r_2$ 
        and with a character $\chi \pmod N$. 
        The evaluations should prove useful for other works.
        
        \begin{mytheo}
        \label{thm:KloostermanCalculation}
             Let $N = pquv$ with $p,q,u,v$ all pairwise coprime. Put $r_1 = pu, s_1 = qv$ and $r_2 = pv, s_2= qu$. 
             The set of allowed moduli for the pair of cusps $1/r_1$, $1/r_2$ is given as
             \begin{equation}\label{eq:generalAtkinLehnerCuspAllowedModuli}
                  \mathcal{C}_{1/r_1, 1/r_2 } = \left\{\kloostermanmodulus =  c\sqrt{uv} > 0: c \equiv 0 \mymod{pq}, \thinspace (c,uv)=1\right\}.
             \end{equation}
             Let $\chi$ be an even character modulo $N$, and factor it as $\chi = \chi_p \chi_q \chi_u \chi_v$ where $\chi_p$ is a
             character modulo $p$,  $\chi_q$ is a character modulo $q$, etc. The Kloosterman sum for this 
             pair of cusps and character $\chi$ with modulus $\kloostermanmodulus = c\sqrt{uv} \in \mathcal{C}_{1/r_1,1/r_2}$ is given as 
             \begin{equation}\label{eq:KloostermanSumAtAtkinLehnerCuspsAndCharacter}
                  S_{1/r_1,1/r_2}  (m,n;c\sqrt{uv} ;\chi) = \mathfrak{f}\, 
                  \overline{\chi_u}(c) \chi_v(c)
                  \sum_{\substack{a, d \shortmod{c}\\ ad\equiv 1 \shortmod{c}}} \overline{\chi_p}(d)\overline{\chi_q}(a) e\left(\frac{a \overline{uv}m  + dn}{c}\right),
             \end{equation}
             where 
             \begin{equation} \label{eq:extraChis}
                  \mathfrak{f}= \mathfrak{f}(p,q,u,v,\chi) = \chi_v(-1) \overline{\chi_p\chi_v} (u) \chi_q\chi_u(v) \chi_u(pq) \overline{\chi_v}(pq).
             \end{equation}
        \end{mytheo}
        
        {\bf Remarks.} One may directly verify that the explicit formula given by  \eqref{eq:KloostermanSumAtAtkinLehnerCuspsAndCharacter} and \eqref{eq:extraChis} also satisfies \eqref{eq:switchCusps} (recall that $\chi$ is even), which is a nice consistency check. Also, both cusps $1/r_1, 1/r_2$ are singular with respect to $\chi$.

        \begin{proof}
             Our proof closely follows that in \cite{motohashi2007riemann}. Consider the double coset $\sigma_{1/r_1}\inv \Gamma \sigma_{1/r_2}$, and
             recall the definitions of $\tau_r$ and $\nu_s$ from \eqref{eq:AtkinLehnerScaling}. We firstly claim 
             \begin{equation} \label{eq:tauDoubleCoset}
                  \tau_{r_1}\inv \Gamma \tau_{r_2} = \left\{\bpm xv & yq \\ zp & wu \epm \in \SL_2(\Z): x,y,z,w \in \Z \right\}.
             \end{equation}
             By reducing the entries of the product of matrices implicit in the left hand side of \eqref{eq:tauDoubleCoset} modulo $p$, $q$, $u$, and $v$ respectively, one 
             sees that in each case the lower-left, upper-right, lower-right, and upper-left entries vanish respectively.  Hence  
             the left hand side of \eqref{eq:tauDoubleCoset} is contained in the set on the right hand side. For the opposite inclusion, we have
             \[
                 \tau_{r_1} \bpm xv & yq \\ zp & wu \epm \tau_{r_2}\inv = \bpm 1&\frac{s_1\overline{s_1}-1}{r_1} \\ r_1& s_1\overline{s_1} \epm \bpm xv & yq \\ zp & wu \epm \bpm s_2\overline{s_2} & \frac{1-s_2\overline{s_2}}{r_2}\\ -r_2 & 1\epm. 
             \]
             Again by reducing modulo $p,q,u$ and $v$ (the reader may find it easiest to reduce prior to performing matrix multiplication), one obtains an upper triangular matrix in each case, whence the product is an element of $\Gamma_0(N) = \Gamma_0(pquv)$.

             Multiplying with the width-normalizing matrices $\nu_s$, we get
             \[
                  \sigma_{1/r_1}\inv \Gamma \sigma_{1/r_2} = \left\{\bpm x\sqrt{uv} &y/\sqrt{uv}\\ zpq \sqrt{uv} & w\sqrt{uv} \epm \in \SL_2(\mr): x,y,z,w \in \Z \right\}.
             \]
             The determinant condition reads as $xwuv - zpqy = 1$, and for this to be satisfied one needs $(z,uv) = 1$. 
             This shows that \eqref{eq:generalAtkinLehnerCuspAllowedModuli} indeed gives the allowable set of moduli.  
             
             Next we wish to decompose this double coset according to the action of $\Gamma_{\infty}$ on both the left and right, as in \cite[Theorem 2.7]{IwaniecSpectralBook}.         
             A full set of coset representatives for $\Gamma_\infty \backslash \sigma_{\frac{1}{r_1}}\inv \Gamma \sigma_{\frac{1}{r_2}} /\Gamma_\infty$ with a given lower-left entry $zpq\sqrt{uv}$ is given by
             \begin{multline}\label{eq:doubleCosetRepresentatives}
                  (\Gamma_\infty \backslash \sigma_{\frac1{r_1}}\inv \Gamma \sigma_{\frac{1}{r_2}} /\Gamma_\infty)
                  \cap \left\{ \left(\begin{smallmatrix} * & * \\ zpq \sqrt{uv} & * \end{smallmatrix} \right) \right\} 
                  \\
                  = \left\{  \bpm x\sqrt{uv} &*\\ zpq \sqrt{uv} & w\sqrt{uv}\epm \in \SL_2(\R): \begin{matrix} x, w \in (\Z/zpq \Z)^* \\ xw uv \equiv 1 \shortmod{zpq} \end{matrix} \right\}.
             \end{multline}
             Here  the condition $xw uv \equiv 1 \pmod{zpq}$   determines $w$ in terms of $x$, and automatically implies $(xw, zpg) = 1$.
              
             Because of the presence of a character, we need to know the lower-right entry of an element of $ \Gamma$ in terms of the integers $x,w,z$ from this double coset.
             Given $\rho = \bsm  x\sqrt{uv} &y/\sqrt{uv}\\ zpq \sqrt{uv} & w\sqrt{uv} \esm$, we compute the lower-right entry of $\bsm *&* \\ * &d \esm = \sigma_{\frac{1}{r_1}} \rho \sigma_{\frac{1}{r_2}}\inv$ 
             by brute-force calculation as
             \begin{align*}
             	d &= \frac{(1-s_2 \overline{s_2})}{r_2} (puxv + s_1 \overline{s_1} zp) + puqy + u s_1 \overline{s_1} w
              \\
              &= (1- qu \overline{qu}) (ux + q \overline{qv} z) + puqy + u qv \overline{qv} w.
             \end{align*}
Reducing this
             in each of the moduli $p,q,u,v$, we obtain
             \begin{align*}
                 d \equiv wu \pmod p, \qquad & \qquad d \equiv ux \pmod{q}, \\
                  d \equiv z\overline{v} \pmod u, \qquad & \qquad d \equiv y puq \equiv -\overline{zpq} puq \equiv -\overline{z}u \pmod v.
             \end{align*}
		Alternatively, one may reduce the matrices prior to the matrix multiplication.            
             
             The Kloosterman sum is then given by 
             \begin{equation*}
                 S_{\tfrac{1}{r_1}, \tfrac{1}{r_2}} (m,n;zpq \sqrt{uv} ;\chi) = \!\!\!\!\!\!\!\! \sum_{\big(\begin{smallmatrix} x\sqrt{uv}&*\\zpq\sqrt{uv} & w\sqrt{uv} \end{smallmatrix} \big) \in \Gamma_\infty \backslash \sigma_{\frac{1}{r_1}}\inv \Gamma \sigma_{\frac{1}{r_2} }/\Gamma_\infty  }
                 \!\!\!\!\!\!\!\!\!\!\!\!\!\!\!\!
                 \overline{\chi_p}(uw) \overline{\chi_q}(ux) \overline{\chi_u}(z\overline{v}) \overline{\chi_v}(-\overline{z}u) e\left(\frac{x m  + w n}{zpq}\right).
             \end{equation*}
             Using \eqref{eq:doubleCosetRepresentatives} and a change of variables $x \mapsto x \overline{uv}$, and with some simplifications, we obtain \eqref{eq:KloostermanSumAtAtkinLehnerCuspsAndCharacter}.
        \end{proof}

 {\bf Examples.}       Specializing Theorem \ref{thm:KloostermanCalculation} to $r_1 = N, r_2 = 1$, we obtain
        \begin{equation} \label{eq:zeroInfinityKloostermanSum}
            S_{\infty,0}  (m,n;c \sqrt{N};\chi) = \overline{\chi}(c) S(\overline{N}m,n;c),
        \end{equation}
        with $(c,N) = 1$.  More generally, we have
        \begin{equation}
            S_{\infty, \frac{1}{r}}  (m,n;c \sqrt{s};\chi)= \overline{\chi_r}(s) \chi_s(r) \overline{\chi_s}(c) \sum_{\substack{ a,d \shortmod{c}\\ ad \equiv 1 \shortmod{c} } } e\left(\frac{a \overline{s}m + d n}{c}\right) \overline{\chi_r}(d),
        \end{equation}
        with $r | c$ and $(c,s) = 1$, and additionally
        \begin{equation}
            S_{0,\frac{1}{r}}  (m,n;c \sqrt{r};\chi) = \chi_r(-1) \chi_s(r) \overline{\chi_r}(s) \chi_r(c) \sum_{\substack{a,d \shortmod{c} \\ ad \equiv 1 \shortmod{c}}} e\left(\frac{a \overline{r}m + dn}{c}\right) \overline{\chi_s}(a),
        \end{equation}
        with $s | c$ and $(c,r) =1$.  These formulas should be contrasted with \cite[p.58]{IwaniecClassicalBook} or \cite[Lemma 4.3]{drappeau2015sums}, which use a different choice of scaling matrices.
       	In \eqref{eq:zeroInfinityKloostermanSum} the occurrence of the factor $\overline{\chi}(c)$ with the \emph{modulus} of the Kloosterman sum is a nice feature of the pair of cusps  $\infty,0$ as opposed to the case
	\begin{equation}
		S_{\infty, \infty} (m,n;c;\chi) = \sum_{ad \equiv 1 \shortmod{c} } e\left(\frac{am + dn}{c}\right) \overline{\chi}(d).
	\end{equation}
	
	{\bf Remark.} 
	A simple application of \eqref{eq:zeroInfinityKloostermanSum} is that 
	for $(a,q) = 1$ we have
	\begin{multline}
	\label{eq:SumsofKloostermanSumsInResidue}
               \sum_{c \equiv a \shortmod{q}} S(m,n;c) f(c) = \frac{1}{\phi(q)}\sum_{\chi \shortmod q} \chi(a)  \sum_{(c,q) =1} \overline{\chi(c)} S(\overline{q} q m,n;c) f(c)\\
               =\frac{1}{\phi(q)}\sum_{\chi \shortmod{q}} \chi(a) \sum_{(c,q) =1 } S_{\infty,0} (qm,n;c\sqrt{q};\chi) f(c).
        \end{multline}
        The sum over $(c,q)=1$ is the set of all allowed moduli for the group $\Gamma_0(q)$ and the $\infty, 0$ cusp pair, so one may directly apply the Bruggeman-Kuznetsov formula to this type of sum.
	The analysis of \eqref{eq:SumsofKloostermanSumsInResidue} was first undertaken by Blomer and Mili\'{c}evi\'{c} 
	\cite{BlomerMilicevic2015KloostermanSumsInResidue}.  Their method makes use of $S_{\infty, \infty}(m,n;c;\chi) $, but they require the 
	use of twisted multiplicativity of Kloosterman sums, and the fact that $S_{\infty,\infty}(m,0;c;\chi)$ is a Gauss sum of the character $\chi$ with twist $m$, in order to spectrally decompose \eqref{eq:SumsofKloostermanSumsInResidue}. 

	
\section{Fourier coefficients of Eisenstein series}
\subsection{Definitions}
Let $\c$ be an arbitrary cusp for $\Gamma_0(N)$.  The Eisenstein series attached to $\c$ (a singular cusp for the nebentypus $\chi$) is defined by
\begin{equation}\label{eq:eisensteinSeriesDefinition}
		E_\c(z,s;\chi)
		= \sum_{\gamma \in \Gamma_\c \backslash \Gamma} \overline{\chi}(\gamma)
		 \Im(\sigma_\c\inv \gamma z)^s.
	\end{equation}

The Fourier expansion around the cusp $\a$ takes the form
\begin{equation}\label{eq:EisensteinFourierExpansion}
		E_\c(\sigma_\a z,u ;\chi) 
		= \delta_{\a\c} y^{u} + \rho_{\a\c}(0,u,\chi) y^{1 - u} 
			+ \sum_{n \neq 0} \rho_{\a\c}(n,u,\chi) e(nx) W_{0,u-\frac12}(4\pi |n| y).
	\end{equation}	 
	Consulting \cite[Theorem 3.4]{IwaniecSpectralBook}, we have
	\begin{equation}
	\label{eq:rhodef}
		\rho_{\a\c}(n,u, \chi)
		= \begin{cases}
		\phi_{\a\c}(n,u, \chi) \frac{\pi^u}{\Gamma(u)} |n|^{u - 1} ,  & \text{ if } n \neq 0
		\\ 
		 \delta_{\a\c}y^u + \phi_{\a\c}(u, \chi) y^{1-u}, & \text{ if } n = 0,
		 \end{cases}
	\end{equation}
	where
	\begin{equation}\label{eq:phiDefinition}
		\phi_{\a\c}(n,u, \chi) 
		= 
		\sum_{\substack{(\gamma,\delta) \text{ such that} \\ \rho =  \bsm *&*\\\gamma&\delta \esm \in  \Gamma_\infty \backslash \sigma_\c\inv \Gamma \sigma_\a /\Gamma_\infty  }} 
		\overline{\chi}(\sigma_\c \rho \sigma_\a\inv) 
		\frac{1}{\gamma^{2u}} e\left(\frac{n\delta}{\gamma}\right)
		= \sum_{\gamma \in \mathcal{C}_{\c \a}} \frac{S_{\c  \a}(0,n;\gamma, \chi)}{\gamma^{2u}} 
		,
	\end{equation}
	and $\phi_{\a\c}(u, \chi) = \phi_{\a\c}(0,u, \chi)$.  
Note that our ordering of the cusps in the notation $\rho_{\a \c}, \phi_{\a \c}$ is reversed from that of \cite{IwaniecSpectralBook}, and also that \cite[(3.22)]{IwaniecSpectralBook} should have $\mathcal{S}_{\a \c}(n,0;c)$ in place of $\mathcal{S}_{\a \c}(0,n;c)$ to be consistent with \cite[(2.23)]{IwaniecSpectralBook}.  In case $\chi$ is principal, we shall drop it from the notation.

The goal of this section is evaluate $\phi_{\a \c}(n,u)$ in explicit terms, for which see Theorem \ref{thm:EisensteinFourierCoefficientFormulaDirichletCharacters} below.

\subsection{Cusps, stabilizers, and scaling matrices}
	First we write down representatives from the set of $\Gamma$-equivalency classes of cusps. An explicit parametrization may be found in \cite[Proposition 2.6]{IwaniecClassicalBook}, however we prefer a different choice as follows.  
        \begin{myprop}
        \label{prop:cuspParameterization}
            Every cusp of $\Gamma_0(N)$ is equivalent to one of the form $\b = 1/w$ with 
            $1 \leq w \leq N$. Two cusps of the form $1/w$ and $1/v$ 
            are equivalent to each other if and only if
            \begin{equation}
            \label{eq:cuspRepresentativesCondition}
                (v,N) = (w,N), \quad \text{ and } \quad 
                \frac{v}{(v,N)} \equiv \frac{w}{(w,N)} \pmod{\big((w,N), \tfrac{N}{(w,N)} \big)}.
            \end{equation}
            
            A cusp of the form $p/q$ is equivalent to one of the form $1/w$ with 
            $w \equiv p'q \pmod{N}$ where $p' \equiv p \pmod{ (q,N)}$ and $(p',N) =1$. 
        \end{myprop}
        \begin{proof}
            Let $\b = p/q$ be a cusp. We may take $(p,q) = 1$. Using Bezout's lemma choose $a,b\in\Z$
            such that $ap + bq =1$,  and $(a,N) = 1$.  Let $\gamma = \bigl( \begin{smallmatrix}
a&b\\ c&d
\end{smallmatrix} \bigr) \in  \Gamma_0(N)  $. This ensures that $\gamma \cdot \b$ is a rational number 
	    with numerator equal to $1$.
	   Replacing
	    $c$ by $c + aNt$ and $d$ with $d + bNt$, 
	    we have
	    \[
                \bpm a &b \\ c + aNt & d  + bNt  \epm \cdot \b = \frac{ap + bq}{(c+ aNt)p + (d + bNt)q} 
                = \frac{1}{cp + dq + Nt}.
            \]
            Hence the denominator may be chosen to lie in the interval $[1,N]$. Further
            note that the denominator is congruent to $dq$ modulo $N$. From
            \[
                d \equiv \overline{a} \pmod{N} \qquad \text{ and } \qquad a \equiv \overline{p} \pmod{q},
            \]
            we deduce that $d \equiv p \pmod{(N,q)}$.  Thus  we get the last statement in the 
            proposition.

            We have established that any cusp is equivalent to one of the form $1/w$ with
            $1 \leq w \leq N$. Now 
            suppose $1/v$ and $1/w$ are equivalent.
            Elements of the group $\Gamma$ send relatively prime integer pairs
            to other such pairs, so if
            \[
                \bpm a&b\\Nc&d \epm \cdot \frac{1}{w} = \frac{1}{v},
            \]
            then switching the signs on $a,b,c,d$ if necessary, we have
            \begin{equation}\label{eq:rTosEquation}
                a + bw = 1 \qquad \text{ and } \qquad Nc + dw = v.
            \end{equation}
            The latter equation implies $(v,N) = (dw, N)$. Since $(d,N) = 1$,
            we get that 
            \begin{equation}\label{eq:equivalenceCondition1}
                (v,N) =  (w,N) .
            \end{equation}
            The first equation in \eqref{eq:rTosEquation} implies $a \equiv 1 \pmod w$, which 
            means that $a \equiv 1 \pmod{(w,N)}$. Since the matrix has determinant $1$, then $ad \equiv 1 \pmod N$, and hence $ad \equiv 1 \pmod{(w,N)}$. Therefore
            \begin{equation}
            \label{eq:dcongruence}
                d \equiv 1 \pmod{(w,N)}.
            \end{equation}
            The second equation in \eqref{eq:rTosEquation} gives that 
            $dw \equiv v \pmod N$, equivalently
            \[
                d \frac{w}{(w,N)} \equiv \frac{v}{(v,N)} \pmod {\tfrac{N}{(w,N)} }.
            \]
            Then  \eqref{eq:dcongruence}  implies
            \begin{equation}  \label{eq:equivalenceCondition2}
                \frac{w}{(w,N)} \equiv \frac{v}{(v,N)} \pmod {\big((w,N) , \tfrac{N}{(w,N)}\big)}.
            \end{equation}
           Thus we have shown if $1/w$ and $1/v$ are equivalent, then \eqref{eq:cuspRepresentativesCondition} holds.
          
           Now suppose that $w,v$ satisfy \eqref{eq:cuspRepresentativesCondition}.
           Let
           \begin{equation*}
                v' = \frac{v}{(v,N)}, \quad w' = \frac{w}{(w,N)}, \quad N' = \frac{N}{(w,N)}.
            \end{equation*}
Then $v' \equiv w' \pmod{(N', (N,w))}$, and $(wv', N') = (w,N') = ((N,w), N')$.  Therefore there exist $b,c \in \mathbb{Z}$ so that $v'-w' = b wv' + cN'$.  That is, $v-w = bwv+ cN$.  Define $a=1-bw$, $d=1+bv$, and let $\gamma = \bigl( \begin{smallmatrix}
a&b\\ cN &d
\end{smallmatrix} \bigr)$, where one may check that $\gamma$ has determinant $1$, so $\gamma \in \Gamma_0(N)$.  Finally, one may directly verify that $\gamma \cdot \frac{1}{w} = \frac{1}{v}$.   
        \end{proof}
        
        In the notation of Proposition \ref{prop:cuspParameterization}, call $(w,N)= f$, and $w = uf$. 
   
        \begin{mycoro}
        \label{coro:cuspParameterization}
        		A complete set of representatives for the set of inequivalent cusps of $\Gamma_0(N)$ is given by $\frac{1}{uf}$ where $f$ runs over divisors of $N$, and $u$ runs modulo $(f,N/f)$, coprime to $(f,N/f)$,
        		and where we choose $u$ so that $(u,N) = 1$, after adding a suitable multiple of $(f,N/f)$.
        \end{mycoro}
        The cusp $0$ is equivalent to $1/1$ and $\infty$ is equivalent to $1/N$, furthermore $1/uf$ 
        is equivalent to the cusp $u/f$ provided $u$ is coprime to $N$.
        
        {\bf Remark.} It is not true that cusps of the form $u/f$ and $1/uf$ are always equivalent, even if $(u,f) = 1$. 
        For example let $N  =72$, and $f = 3$. We have $\frac{2}{3} \not\sim_{\Gamma} \frac{1}{6}$, however it is true that $\frac{2}{3} \sim_\Gamma \frac{5}{ 3} \sim_\Gamma \frac{1}{15}$.
        

We need to compute the stabilizers of various cusps. 
        \begin{myprop}\label{prop:stabilizerScaling}
            Let $\c= 1/w$ be a cusp of $\Gamma = \Gamma_0(N)$, and set
            \begin{equation}
            \label{eq:Nandwformulas}
              N = (N,w) N'   \qquad w = (N,w) w', 
              \qquad
              N' = (N',w) N''.  
            \end{equation}
            The stabilizer of $1/w$ is given as 
            \begin{equation}
            \label{eq:stabilizerFormula}
                \Gamma_{1/w} = \left\{ \pm \lambda_{1/w}^t  : t \in \Z \right\}, \quad \text{where} \quad \lambda_{1/w}^t = \bpm 1 - wN'' t  & N'' t \\ -w^2 N'' t & 1+ wN'' t \epm,
            \end{equation}
            and one may choose the scaling matrix as
           \begin{equation}\label{eq:anyCuspScaling}
               \sigma_{1/w} = \bpm 1&0\\ w&1 \epm  \bpm \sqrt{N''}  &0\\ 0 &1/\sqrt{N''} \epm .
           \end{equation}
        \end{myprop}
        
	{\bf Remark.}  
	This is not the choice of scaling matrix we made in Section \ref{subsec:AtkinLehnerScaling} for an Atkin-Lehner cusp. When computing the Fourier coefficients of $E_{\c}(\sigma_\a z, u)$, with $\a=1/r$ Atkin-Lehner and $\c$ arbitrary, we will choose \eqref{eq:anyCuspScaling} for $\sigma_\c$, and \eqref{eq:AtkinLehnerScaling} for $\sigma_\a$. In addition, one should observe that $N | w^2 N''$ to see that $\lambda_{1/w} \in \Gamma_0(N)$. As an aside, note that $N''$ is the width of the cusp $1/w$. 

        \begin{proof}
       Taking $\gamma = \bigl( \begin{smallmatrix}
a&b\\ cN &d
\end{smallmatrix} \bigr) \in \Gamma_0(N)$ so that $\gamma \cdot \frac{1}{w} = \frac{1}{w}$ means that $a+bw = \pm 1$, and $cN+dw = \pm w$ (with the same choice of $\pm$ sign).  Consider the $+$ sign case, the $-$ sign following by symmetry.  The former equation determines $a$ in terms of $b$, and the latter is equivalent to $cN' + dw' = w'$.  
Since $(N',w') = 1$, we have $w'|c$, so define $c=w'r$.
Then $d=1-rN'$.  Finally, the condition that $\gamma$ has determinant $1$ means $bw=-rN'$, which means $b=N''t$ (say) and $r=-\frac{w}{(w,N')} t$, giving that $\gamma$ takes the form as stated in \eqref{eq:stabilizerFormula}.  One may conversely check that any matrix of the form stated in \eqref{eq:stabilizerFormula} stabilizes $1/w$.

To show the final statement, one easily calculates that
$\sigma_{1/w}^{-1} \lambda_{1/w} \sigma_{1/w} = (\begin{smallmatrix} 1 & 1 \\ 0 & 1 \end{smallmatrix})$,
so \eqref{eq:scalingMatrixProperties} is satisfied.
        \end{proof}

\subsection{Statement of result}

	We compute the Fourier coefficients of $E_\c(z,u)$ at an Atkin-Lehner cusp $\a =1/r$ with $rs=N$, $(r,s) = 1$.  Write $\c = \frac{1}{w} =  \frac{1}{vf}$ where $f|N$ and $(v,N) = 1$.   
Let
\begin{equation}
\label{eq:fN''formula}
 N' = \frac{N}{f}, \qquad N'' = \frac{N'}{(f,N')},
\end{equation}
and write
\begin{equation}
\label{eq:frfsr's'Definitions}
 f_r = (f,r), \quad f_s = (f,s), \quad
 r = f_r r', \quad s= f_s s'.
\end{equation}
In addition, write
\begin{equation*}
 f_r = f_r' f_0, \quad \text{where} \quad (f_0, r') = 1, \quad \text{and} \quad f_r' | (r')^{\infty},
\end{equation*}
and similarly 
\begin{equation*}
 s' = s_f' s_0, \quad \text{where} \quad (s_0, f_s) = 1, \quad \text{and} \quad s_f' | f_s^{\infty}.
\end{equation*}

\begin{mytheo}
\label{thm:EisensteinFourierCoefficientFormulaDirichletCharacters}
 Let notation be as above.  Then $\phi_{\a \c}(n,u) = 0$ unless
 \begin{equation*}
n=\frac{f_r'}{(f_r',r')} \frac{s_f'}{(s_f',f_s)} k,
\end{equation*}
for some integer $k$.  In this case, write $k = k_r k_s \ell$, where 
\begin{equation*}
 k_r = (k, (f_r', r')), \qquad k_s = (k, (s_f', f_s)).
\end{equation*}
 Then
 \begin{multline}
\label{eq:phicuspsEisensteinFormulaWithDirichletCharacters}
\phi_{\a \c}(n,u) = \frac{S(\ell,0;s_0 f_0)}{(N'' s f_r^2)^u} 
\frac{f_r'}{(f_r', r')} 
\frac{s_f'}{(s_f', f_s)} 
\sum_{\substack{d|k \\ (d, f_s r') = 1}} d^{1-2u}
\frac{1}{\varphi(\frac{(f_r', r')}{k_r})}  
\frac{1}{\varphi(\frac{(s_f', f_s)}{k_s})}
\\
\sum_{\chi \shortmod{\frac{(f_r', r')}{k_r}}} 
  \sum_{\psi \shortmod{\frac{(s_f', f_s)}{k_s}}} 
  \frac{(\chi \psi)(\ell)
  \tau(\overline{\chi}) \tau(\overline{\psi})}{L(2u, \overline{\chi^2 \psi^2} \chi_0)}
  (\chi \psi)(\overline{s_0 f_0  d^2} v) 
  \chi(-k_s\overline{(s_f',f_s)})
  \psi(k_r \overline{(f_r',r')})
,
\end{multline}
where $\chi_0$ is the principal character modulo $f_s r'$.
\end{mytheo}

The proof of Theorem \ref{thm:EisensteinFourierCoefficientFormulaDirichletCharacters} takes up the rest of the paper.

\subsection{Double cosets}
To begin, we obtain an explicit formula for the double coset appearing in \eqref{eq:phiDefinition}.	
	\begin{mylemma}
	\label{lemma:DoubleCosetFormula}
		Let $\c = 1/w$ be any cusp of $\Gamma=\Gamma_0(N)$ and $\a = 1/r$ an Atkin-Lehner cusp. 
		Let the scaling matrices be as in \eqref{eq:AtkinLehnerScaling} and 
		\eqref{eq:anyCuspScaling}. Then 
		\begin{equation}
		\label{eq:doublecoset}
			\sigma_\c\inv \Gamma \sigma_\a 
			= \left\{\bpm \frac{A}{N''} \sqrt{N'' s}   & \frac{B}{N''} \sqrt{N''/s}  \\ C \sqrt{N''s} &D\sqrt{N''/s}  \epm: 
			\bpm A & B \\ C & D \epm \in \SL_2(\Z),\,\,
			\begin{matrix}
			C \equiv - wA \mymod{r} \\ D \equiv - w B \mymod{s} 
			\end{matrix}
			\right\}.
		\end{equation}
	\end{mylemma}
	
	\begin{proof}
		Let us call $\tau_{\c} = \bsm 1&0\\w&1\esm$, and $\nu_\c = \bsm \sqrt{N''} & 0 \\ 0& 1/\sqrt{N''} \esm$, so 
		$\sigma_\c = \tau_\c \nu_\c$. Let $\tau_r \nu_s$
		denote the decomposition of the scaling matrix in \eqref{eq:AtkinLehnerScaling}. Take 
		$\bsm a&b\\Nc&d \esm \in \Gamma$ and compute
		\begin{equation}\label{eq:doubleIntegralCoset}
			\bpm A&B \\ C &D \epm 
			= 
			\tau_\c\inv \bpm a&b\\Nc&d\epm \tau_r 
			= \bpm 1&0\\ -w &1 \epm \bpm a&b\\Nc&d \epm \bpm 1&\frac{s\overline{s} -1 }{r}\\ r& s\overline{s} \epm
			\in \SL_{2}(\Z).
		\end{equation}
		Note that
		\begin{equation*}
		 \nu_{\c}^{-1} \bpm A&B \\ C &D \epm \nu_s = \bpm \frac{A}{N''} \sqrt{N'' s}   & \frac{B}{N''} \sqrt{N''/s}  \\ C \sqrt{N''s} &D\sqrt{N''/s}  \epm.
		\end{equation*}

		Considering the product modulo $r$ gives 
		\[
			\bpm A&B\\ C & D \epm 
			\equiv \bpm 1&0\\-w&1 \epm \bpm a&b\\0&d\epm \bpm 1&*\\0&1 \epm 
			\equiv \bpm a&*\\ -aw &*\epm \pmod r,
		\]
		and hence $C \equiv - w A \pmod{r}$.
		Reducing modulo $s$, we obtain
		\[
			\bpm A&B\\ C & D \epm  
			\equiv \bpm 1&0\\ -w&1 \epm \bpm a&b\\0&d \epm \bpm 1&-\overline{r} \\ r & 0\epm 
			\equiv \bpm * &-a\overline{r}\\ * & wa\overline{r} \epm \pmod s,
		\]
		so $D \equiv - w B \pmod{s}$.
		
		Now we check that given $\bsm A&B\\ C & D \esm$ satisfying
		the conditions in \eqref{eq:doublecoset}, then it is covered by the products of the form in \eqref{eq:doubleIntegralCoset}. 
		For that purpose we compute 
		\[
			\tau_\c \bpm A & B \\ C &D \epm \tau_r\inv  
			= \bpm 1&0\\w&1 \epm \bpm A&B\\ C & D \epm \bpm s \overline{s} &\frac{1-s\overline{s}}{r} \\ -r &1\epm.
		\]
		Modulo $r$, the lower left entry of this product is congruent to $wA + C \equiv 0 \pmod r$, 
		and also congruent to $-B rw - D r = -r(B w + D) \equiv 0 \pmod s$. This implies that the
		lower left entry is divisible by $N$.
	\end{proof}

The next step towards evaluation of \eqref{eq:phiDefinition} is to work out representatives for $\Gamma_{\infty} \backslash \sigma_\c\inv \Gamma \sigma_\a / \Gamma_{\infty}$.  As a consistency check, note that the action of $\Gamma_{\infty}$ on both the left and right does not affect the congruences linking $A$ to $C$ and $B$ to $D$, which ultimately follows from $w^2 N'' \equiv 0 \pmod{N}$.
We need to find the set of pairs $C$, $D$ with $C > 0$, $(C,D) = 1$, $D \pmod{sC}$, for which there exist integers $A$, $B$ with $AD-BC=1$ and so that $C \equiv - w A \pmod{r}$ and $D \equiv -wB \pmod{s}$.

Before stating the result, we develop some notation.
Suppose that $A,B,C,D$ are as in the right hand side of \eqref{eq:doublecoset}.
Recall the notation from \eqref{eq:fN''formula}, \eqref{eq:frfsr's'Definitions}, recall $w = vf$, $(v,N) = 1$,
and note $f = f_r f_s$.   From $(A,C) = 1$ and $C \equiv - w A \pmod{r}$, we derive
$(C,r) = (w,r) = f_r$.  Similarly, $(D,s) = (w,s) = f_s$. Write 
\begin{equation}
\label{eq:C'D'r's'Definitions}
C = f_r C', \qquad r = f_r r', \qquad D = f_s D', \qquad s= f_s s',
\end{equation}
where $(C',r') = 1 = (D',s')$.
 Then we have
\begin{equation}
\label{eq:ABCDcongruences}
\begin{matrix}
			C \equiv - wA \mymod{r} \\ D \equiv - w B \mymod{s} 
			\end{matrix}
	\Longleftrightarrow
\begin{matrix}
			A \equiv - \overline{({w}/{f_r})} C' \mymod{r'} 
			\\ B \equiv  - \overline{({w}/{f_s})} D' \mymod{s'}.
			\end{matrix}		
\end{equation}
The equivalence of congruences in \eqref{eq:ABCDcongruences} only uses the assumptions $(A,C) = (B,D) = 1$.

From $AD-BC =1$, we have $D \equiv \overline{A} \pmod{|C|}$, which combined with the right hand side of \eqref{eq:ABCDcongruences} gives
\begin{equation*}
D \equiv - \overline{C'} \frac{w}{f_r}  \pmod{(f_r, r')}, 
\end{equation*}
using $(C,r') = (f_r C', r') = (f_r,r')$.  Similarly, we have $B \equiv - \overline{C} \pmod{|D|}$, so $D' \equiv \overline{C} \frac{w}{f_s}    \pmod{(f_s, s')}$. 

\begin{mylemma}
With notation as in Lemma \ref{lemma:DoubleCosetFormula} and its following discussion, we have the disjoint union
\begin{multline}
\label{eq:DoubleCosetRepresentativesLowerRow}
\Gamma_{\infty} \backslash \sigma_\c\inv \Gamma \sigma_\a / \Gamma_{\infty}
=
\delta_{\c \a} \Omega_{\infty}
\\
\cup 
\left\{ 
\begin{pmatrix}
* & * \\
C \sqrt{N'' s} & \frac{D}{s} \sqrt{N'' s} 
\end{pmatrix}
:
\begin{matrix}
			(C,r) = f_r \\ (D,s) =  f_s
			\end{matrix},
\quad
\begin{matrix}
D \equiv - \overline{C'} \frac{w}{f_r} \mymod{(f_r,r')}
\\
D' \equiv \overline{C} \frac{w}{f_s}  \mymod{(f_s, s')}
\end{matrix}
\right\},
\end{multline}
where in addition we have the restrictions $C > 0$,		
 $(C,D) = 1$ and $D$ runs over representatives $\pmod{sC}$.
  Here
$\delta_{\c \a} \Omega_{\infty}$ denotes $\Gamma_{\infty}$ if $\c$ is equivalent to $\a$, and denotes the empty set if $\c$ is not equivalent to $\a$.
  \end{mylemma}

\begin{proof}
The discussion preceding the statement of the lemma shows that any matrix given on the right hand side of \eqref{eq:doublecoset} with $C >0$ gives rise to a double coset of the claimed form.  It suffices to show that given integers $C,D$ as in the second line of \eqref{eq:DoubleCosetRepresentativesLowerRow}, we may find $A,B$ so that $AD-BC=1$ and satisfying the congruences in \eqref{eq:ABCDcongruences}.

From $(C,D) = 1$, we may choose $A_0, B_0$ so that $A_0 D - B_0 C = 1$.  From $A_0 \equiv \overline{D} \pmod{C}$, and the congruence on $D$ given in \eqref{eq:DoubleCosetRepresentativesLowerRow}, we have
\begin{equation*}
A_0 \equiv - (\overline{w/f_r}) C' \pmod{(f_r,r')}.
\end{equation*}
Hence, there exist integers $x,y$ so that $A_0 + C' \overline{(w/f_r)} = f_r x + r' y$.   A similar argument with $B_0$ gives
\begin{equation*}
B_0 \equiv - (\overline{w/f_s}) D' \pmod{(f_s, s')},
\end{equation*}
and so there exist integers $X,Y$ so that $B_0 + D' \overline{(w/f_s)} = f_s X + s' Y$.

We next want to find $n \in \mathbb{Z}$ so that $A = A_0 + nC = A_0 + nC' f_r$ and $B = B_0 + nD = B_0 + nD' f_s$ 
satisfies the right hand side of \eqref{eq:ABCDcongruences}.  Gathering the above formulas, we have
\begin{align*}
A = A_0 + nC &= -\overline{(w/f_r)} C' + f_r(x+nC') + r'y \\
B = B_0 + nD &= -\overline{(w/f_s)} D' + f_s(X+nD') + s'Y.
\end{align*}
We may choose $n$ so that $n \equiv -\overline{C'}x \pmod{r'}$ and $n \equiv - \overline{D'} X \pmod{s'}$, since $(C',r') = 1 = (D',s') =(r', s')$, and by the Chinese remainder theorem.  With this choice of $n$, then $A$ and $B$ satisfy the congruences on the right hand side of \eqref{eq:ABCDcongruences}.
\end{proof}

Finally we can evaluate $\phi_{\a \c}(n,u)$.  First, note that the congruence $D \equiv - \overline{C'} \frac{w}{f_r} \pmod{(f_r, r')}$ is equivalent to $D' \equiv - \overline{C'} \overline{f_s} \frac{w}{f_r} \pmod{(f_r, r')}$, and that we can write $w = f_r f_s w'$, giving now $D' \equiv -\overline{C'} w' \pmod{(f_r, r')}$.  Similarly, the other congruence in \eqref{eq:DoubleCosetRepresentativesLowerRow} is equivalent to $D' \equiv  \overline{C'} w' \pmod{(f_s, s')}$.  

Putting everything together, we have
\begin{equation}
\label{eq:phiacCalculation}
\phi_{\a \c}(n,u) = \frac{1}{(N'' s f_r^2)^u} \sum_{(C', f_s r') = 1} \frac{1}{(C')^{2u}} 
\sumstar_{\substack{D' \shortmod{s' f_r C'} 
	\\  D' \equiv - \overline{C'} w' \shortmod{(f_r, r')} 
	\\   D' \equiv  \overline{C'} w' \shortmod{(f_s, s')} 
	}} e\Big(\frac{nD'}{s'f_r C'}\Big).
\end{equation}
Now write 
\begin{equation*}
 f_r = f_r' f_0, \quad \text{where} \quad (f_0, r') = 1, \quad \text{and} \quad f_r' | (r')^{\infty},
\end{equation*}
and similarly 
\begin{equation*}
 s' = s_f' s_0, \quad \text{where} \quad (s_0, f_s) = 1, \quad \text{and} \quad s_f' | f_s^{\infty}.
\end{equation*}
Then $(f_r, r') = (f_r',r')$, and $(f_s, s') = (f_s, s_f')$, and so
\begin{equation*}
\phi_{\a \c}(n,u) = \frac{1}{(N'' s f_r^2)^u} \sum_{(C', f_s r') = 1} \frac{1}{(C')^{2u}} 
\sumstar_{\substack{D' \shortmod{s_0 f_0 C' f_r' s_f'} 
	\\  D' \equiv - \overline{C'} w' \shortmod{(f_r', r')} 
	\\   D' \equiv  \overline{C'} w' \shortmod{(f_s, s_f')} 
	}} e\Big(\frac{nD'}{s_0 f_0 C' f_r' s_f'}\Big).
\end{equation*}
By the Chinese remainder theorem, this factors as
\begin{multline*}
\phi_{\a \c}(n,u) = \frac{1}{(N'' s f_r^2)^u} \sum_{(C', f_s r') = 1} \frac{S(n,0;C' s_0 f_0)}{(C')^{2u}} 
\\
\times \Big(
\sumstar_{\substack{D' \shortmod{f_r'} 
	\\  D' \equiv - \overline{C'} w' \shortmod{(f_r', r')}  
	}} e\Big(\frac{nD' \overline{s_0 f_0 C' s_f'}}{ f_r' }\Big)
	\Big)
	\Big(
\sumstar_{\substack{D' \shortmod{s_f'} 
	\\  D' \equiv  \overline{C'} w' \shortmod{(f_s, s_f')}  
	}} e\Big(\frac{nD' \overline{s_0 f_0 C' f_r'}}{ s_f' }\Big)	
	\Big)
	.
\end{multline*}
For the sum modulo $f_r'$, note that $(D', f_r') = 1$ if and only if $(D', (f_r', r')) = 1$, since $f_r' | r'^{\infty}$.  Therefore, the congruence automatically implies $(D', f_r') = 1$, and so this condition may be dropped.  Then after changing variables, it becomes a linear exponential sum which is easy to evaluate.  A similar discussion holds for the modulus $s_f'$.  In this way, we obtain
\begin{multline*}
\phi_{\a \c}(n,u) = \frac{1}{(N'' s f_r^2)^u} 
\frac{f_r'}{(f_r', r')} 
\delta\Big(\frac{f_r'}{(f_r',r')} | n\Big)
\frac{s_f'}{(s_f', f_s)} 
\delta\Big(\frac{s_f'}{(s_f', f_s)} | n\Big)
\\
\sum_{(C', f_s r') = 1} \frac{S(n,0;C' s_0 f_0)}{(C')^{2u}} 
e\Big(\frac{-n w' \overline{ s_0 f_0 s_f'   C'^2}  }{f_r'} \Big)  
e\Big(\frac{n w' \overline{  s_0  f_0 f_r' C'^2} }{s_f'} \Big) .
\end{multline*}
Now write 
\begin{equation*}
n=\frac{f_r'}{(f_r',r')} \frac{s_f'}{(s_f',f_s)} k,
\end{equation*}
and note $(\frac{f_r'}{(f_r', r')} \frac{s_f'}{(s_f',f_s)}, C' s_0 f_0) = 1$
giving
\begin{multline*}
\phi_{\a \c}(n,u) = \frac{S(k,0;s_0 f_0)}{(N'' s f_r^2)^u} 
\frac{f_r'}{(f_r', r')} 
\frac{s_f'}{(s_f', f_s)} 
\\
\sum_{(C', f_s r') = 1} \frac{S(k,0;C')}{(C')^{2u}} 
e\Big(\frac{-k w' \overline{ s_0 f_0 (s_f', f_s)   C'^2}  }{(f_r',r')} \Big)  
e\Big(\frac{k w' \overline{ s_0  f_0(f_r', r')  C'^2} }{(s_f',f_s)} \Big) .
\end{multline*}
Next we use the evaluation of the Ramanujan sum as a divisor sum, giving
\begin{multline*}
\phi_{\a \c}(n,u) = \frac{S(k,0;s_0 f_0)}{(N'' s f_r^2)^u} 
\frac{f_r'}{(f_r', r')} 
\frac{s_f'}{(s_f', f_s)} 
\sum_{\substack{d|k \\ (d, f_s r') = 1}} d^{1-2u}
\\
\sum_{(C', f_s r') = 1} \frac{\mu(C')}{(C')^{2u}} 
e\Big(\frac{-k w'\overline{ s_0 f_0 (s_f', f_s) d^2   C'^2}   }{(f_r',r')} \Big)  
e\Big(\frac{k w' \overline{ s_0  f_0(f_r', r') d^2  C'^2}   }{(s_f',f_s)} \Big) .
\end{multline*}

Our next goal is to change the additive character into multiplicative ones.  This requires placing the fractions in lowest terms.

Tracing back the definitions, we may check that $(w', (f_r', r')) = 1$, since this is equivalent to $(w', r') = 1$, which is true by the definitions of $w'$ and $r'$.  Similarly, $(w', (s_f', f_s)) = 1$.  Next write $k = k_r k_s \ell$, where
\begin{equation*}
 k_r = (k, (f_r', r')), \qquad k_s = (k, (s_f', f_s)).
\end{equation*}
After this, we may convert to Dirichlet characters (see \cite[(3.11)]{iwaniec2004analytic}).  In all, we obtain
\begin{multline}
\phi_{\a \c}(n,u) = \frac{S(\ell,0;s_0 f_0)}{(N'' s f_r^2)^u} 
\frac{f_r'}{(f_r', r')} 
\frac{s_f'}{(s_f', f_s)} 
\sum_{\substack{d|k \\ (d, f_s r') = 1}} d^{1-2u}
\frac{1}{\varphi(\frac{(f_r', r')}{k_r})}  
\frac{1}{\varphi(\frac{(s_f', f_s)}{k_s})}
\\
\sum_{\chi \shortmod{\frac{(f_r', r')}{k_r}}} 
  \sum_{\psi \shortmod{\frac{(s_f', f_s)}{k_s}}} 
  \frac{(\chi \psi)(\ell)
  \tau(\overline{\chi}) \tau(\overline{\psi})}{L(2u, \overline{\chi^2 \psi^2} \chi_0)}
  (\chi \psi)(\overline{s_0 f_0  d^2} w') 
  \chi(-k_s\overline{(s_f',f_s)})
  \psi(k_r \overline{(f_r',r')})
,
\end{multline}
where $\chi_0$ is the principal character modulo $f_s r'$.  This completes the proof of Theorem \ref{thm:EisensteinFourierCoefficientFormulaDirichletCharacters}.

\bibliographystyle{amsalpha}
\bibliography{MehmetBib3}

\end{document}